\documentclass[11 pt]{amsart}
\usepackage{amscd,amsfonts,amssymb,amsmath}
\usepackage{hyperref}
\usepackage{epsfig}
\newtheorem{theorem}{Theorem}[section]
\newtheorem{corollary}[theorem]{Corollary}

\newtheorem{proposition}[theorem]{Proposition}
\theoremstyle{definition}
\newtheorem{conjecture}[theorem]{Conjecture}
\newtheorem{question}[theorem]{Question}

\newtheorem{example}[theorem]{Example}
\newtheorem{remark}[theorem]{Remark}
\numberwithin{equation}{subsection}

\usepackage[all,cmtip]{xy}

\newcommand{\Aut}{\operatorname{Aut}}

\newcommand{\Conj}{\operatorname{Conj}}
\newcommand{\Core}{\operatorname{Core}}

\newcommand{\Inn}{\operatorname{Inn}}

\newcommand{\T}{\operatorname{T}}

\newcommand{\R}{\operatorname{R}}

\newcommand{\Z}{\operatorname{Z}}

\newcommand{\id}{\mathrm{id}}

\begin{document}
\title{Quandle rings}
\author{Valeriy G. Bardakov}
\author{Inder Bir S. Passi}
\author{Mahender Singh}

\date{\today}
\address{Sobolev Institute of Mathematics and Novosibirsk State University, Novosibirsk 630090, Russia.}
\address{Novosibirsk State Agrarian University, Dobrolyubova street, 160, Novosibirsk, 630039, Russia.}
\email{bardakov@math.nsc.ru}

\address{Center for Advanced Study in Mathematics, Panjab University, Chandigarh 160014, India.}
\address{Department of Mathematical Sciences, Indian Institute of Science Education and Research (IISER) Mohali, Sector 81,  S. A. S. Nagar, P. O. Manauli, Punjab 140306, India.}
\email{passi@iisermohali.ac.in}

\address{Department of Mathematical Sciences, Indian Institute of Science Education and Research (IISER) Mohali, Sector 81,  S. A. S. Nagar, P. O. Manauli, Punjab 140306, India.}
\email{mahender@iisermohali.ac.in}


\begin{abstract}
In this paper, a theory of quandle rings is proposed for quandles analogous to the classical theory of group rings for groups, and interconnections between quandles and associated quandle rings are explored. 
\end{abstract}
\maketitle

\section{Introduction}
A quandle is a set with a binary operation that satisfies three axioms motivated by the three Reidemeister moves of diagrams of knots in the Euclidean space $\mathbb{R}^3$.  Ignoring the first Reidemeister move gives rise to a weaker structure called a rack. These algebraic objects were introduced independently by Matveev \cite{Matveev} and Joyce \cite{Joyce-thesis, Joyce}. They associated a quandle to each tame knot in $\mathbb{R}^3$ and showed that it is a complete invariant up to orientation. Over the years, racks, quandles and their analogues have been investigated to construct  more  computable invariants for knots and links, as well as purely algebraic objects, which is also the take of this paper.
\par

Historically, special kind of quandles, now called involutary quandles, had already appeared in the work of Takasaki \cite{Takasaki} on finite geometry. In an unpublished 1959 correspondence, Wraith and Conway \cite{Carter} discussed groups acting on themselves via conjugation yielding quandles now called conjugation quandles. In \cite{Loos1, Loos2}, Loos defined a differentiable reflection space to be a differentiable manifold with a differentiable binary operation satisfying the quandle axioms. For example, if $G$ is a connected Lie group with an involutive automorphism $\varphi$ and $H$ a subgroup of the fixed-point subgroup containing the identity component, then the symmetric space $G/H$ acquires the structure of differentiable reflection space with the binary operation $$g_1H *g_2H=\varphi(g_1 g_2^{-1})g_2H.$$ Conversely, any differentiable reflexion space can be obtained from a symmetric space \cite{Loos1, Loos2}. Results of this kind have led to transfer of ideas from the theory of Riemannian symmetric spaces to the theory of quandles \cite{Ishihara}. Interestingly, quandles appear naturally in many other contexts. Joyce \cite{Joyce-thesis} related quandles to the theory of loops by observing that the core of a Moufang loop is an involutary quandle. Further, quandles have been related to pointed Hopf algebras \cite{Andruskiewitsch}, set-theoretic solutions to the Yang-Baxter equations and Yetter-Drinfeld Modules  \cite{Eisermann}, categorical groups and other notions of categorification \cite{CCES1, CCES2}. We refer the reader to the survey articles \cite{Carter, Kamada, Nelson} for more on the historical development of the subject and its relationships with other areas of mathematics.
\par

Naturally, quandles and racks have also received a great deal of attention as purely algebraic objects. As is the case with any algebraic system, a (co)homology theory for quandles and racks has been developed in \cite{Carter2, Fenn1, Fenn2, Nosaka}, which, as applications has led to stronger invariants for knots and links. An explicit description of abelian group objects in the category of quandles and racks has been given in \cite{Jackson}, leading to the construction of an abelian category of modules over these objects. Automorphisms of quandles have been investigated in much detail. In \cite{Ho-Nelson},  automorphism groups of quandles of order less than 6 were determined. This investigation was carried forward in \cite{Elhamdadi}, wherein the automorphism group of the dihedral quandle $\R_n$ was shown to be isomorphic to the group of invertible affine transformations of $\mathbb{Z}_n$, and the inner automorphism group of $\R_n$ was shown to be isomorphic to some dihedral group. These descriptions of automorphism groups were used to determine, up to isomorphism,  all quandles of order less than 10. In \cite{Hou}, a description of the automorphism group of Alexander quandles was determined, and an explicit formula for the order of the automorphism group was given for finite case. In \cite{Bardakov}, some structural results are obtained for the group of automorphisms and inner automorphisms of generalised Alexander quandles of finite abelian groups with respect to fixed-point free automorphisms. This work was extended in \cite{BarTimSin}, wherein several interesting subgroups of automorphism groups of conjugation quandles of groups are determined. Further, necessary and sufficient conditions are found for these subgroups to coincide with the full group of automorphisms.
\par

The purpose of the present paper is to develop a theory of quandle rings analogous to the classical theory of group rings. Attempt has been made to state the results making distinction between quandles and racks whenever possible. The main results of the paper are Theorems \ref{deltasqzero},  \ref{ideal-to-subquandle}, \ref{id-subq}, \ref{dictionary}, \ref{split-short-seq},  \ref{unit-trivial-rack} and Propositions  \ref{ass-graded-1}, \ref{ass-graded-2}, \ref{ass-graded-3}, \ref{power-ass-1} and \ref{power-ass-2}. Though our approach and motivation is purely algebraic, the results might be of use in knot theory.
\par
The paper is organised as follows. In Section \ref{section2}, we recall some basic definitions and examples from the theory of racks and quandles. 
\par

In Section \ref{section3}, the main objects of our study are introduced. Given a quandle $X$ and an associative ring $R$ (not necessarily with unity), the quandle ring $R[X]$ of $X$ with coefficients in $R$ is defined as the set of formal finite $R$-linear combinations of elements of $X$. Rack rings are defined analogously. In the first main result, Theorem \ref{deltasqzero}, it is proved that a quandle $X$ is trivial if and only if $\Delta^2_R(X) = \{0\}$, where $\Delta_R(X)$ is the augmentation ideal of $R[X]$. While this characterisation holds for quandles, an example is given to show that it fails for racks.
\par

In Section \ref{section4}, relationships between subquandles of the given quandle and ideals of the associated quandle ring are discussed. Given a quandle $X$, an element $x_0 \in X$ and a two sided ideal $I$ of $R[X]$, a subquandle $X_{I, x_0}$ of $X$ is defined. It is shown in Theorem \ref{ideal-to-subquandle} that  these subquandles give a partition of the quandle $X$. Further, relationships among these subquandles are explored. In Theorem \ref{id-subq}, it is shown that if $X$ is a finite involutary quandle, $I$ a two-sided ideal of $R[X]$ and $x_0, y_0 \in X$ two elements in the same orbit under the action of $\Inn(X)$, then $X_{I, x_0} \cong X_{I, y_0}$. Given a surjective quandle homomorphism $f : X \to Z$, in Theorem \ref{quotient-quandle}, it is shown that $Z$ is isomorphic as a quandle to a natural quotient of $X$. The construction of the quotient quandle leads to Theorem \ref{dictionary}, which gives a correspondence between subquandles of the given quandle and ideals of the quandle ring.
\par
Since the quandle ring $R[X]$ does not have unity, it is desirable to embed $R[X]$ into a ring with unity. In Section \ref{section5}, the extended quandle ring $R^\circ[X]$ is introduced, which is a ring with unity containing the ring $R[X]$ as a subring. In Theorem \ref{split-short-seq}, a short exact sequence relating certain subgroups of a group of units of $R^\circ[X]$ is derived. Further, in Theorem \ref{unit-trivial-rack}, the structure of unit groups of $R^\circ[X]$ for trivial quandles $X$ is described.
\par
Given a quandle $X$ and a ring $R$,  the direct sum
$$
\mathcal{X}_R(X) := \sum_{i\geq 0} \Delta^i_R(X) / \Delta^{i+1}_R(X)
$$
of $R$-modules $\Delta^i_R(X) / \Delta^{i+1}_R(X)$ becomes a graded ring, called the associated graded ring of $R[X]$. In Section \ref{section6}, the quotients $\Delta^i_R(X) / \Delta^{i+1}_R(X)$ are investigated for dihedral quandles, and some structural results are obtained in Propositions \ref{ass-graded-1}, \ref{ass-graded-2} and \ref{ass-graded-3}.
\par
Finally, in Section \ref{section7}, we investigate weaker forms of associativity in quandle rings.  A ring is called power-associative if every element of the ring generates an associative subring. In Propositions \ref{power-ass-1} and \ref{power-ass-2}, we prove that quandle rings of dihedral quandles are not power-associative in general.
\bigskip

\section{Preliminaries}\label{section2}
A {\it rack} is a non-empty set $X$ with a binary operation $(x,y) \mapsto x * y$ satisfying the following axioms:
\begin{enumerate}
\item[(R1)] For any $x,y \in X$ there exists a unique $z \in X$ such that $x=z*y$;
\item[(R2)] $(x*y)*z=(x*z) * (y*z)$ for all $x,y,z \in X$.
\end{enumerate}

A rack is called a {\it quandle} if the following additional axiom is satisfed:
\begin{enumerate}
\item[(Q1)] $x*x=x$ for all $x \in X$.
\end{enumerate}
The axioms (R1), (R1) and (Q1) are collectively called quandle axioms. Besides knot quandles associated to knots, many interesting examples of quandles come from groups. Throughout the paper, we write arbitrary groups multiplicatively and abelian groups additively.

\begin{itemize}
\item If $G$ is a group, then the set $G$ equipped with the binary operation $a*b= b^{-1} a b$ gives a quandle structure on $G$, called the {\it conjugation quandle}, and denoted by $\Conj(G)$.
\item If $A$ is an additive abelian group, then the set $A$ equipped with the binary operation $a*b= 2b-a$ gives a quandle structure on $A$, denoted by $T(A)$ and called the Takasaki quandle of $A$. For $A= \mathbb{Z}/n \mathbb{Z}$,  it is called the {\it dihedral quandle}, and is denoted by $\R_n$.
\item If $G$ is a group and we take the binary operation $a*b= b a^{-1} b$, then we get the {\it core quandle}, denoted as $\Core(G)$. In particular, if $G$ is additive abelian, then $\Core(G)$ is the Takasaki quandle.
\item Let $A$ be an additive abelian group and $t \in \Aut(A)$. Then the set $A$ equipped with the binary operation $a* b=ta+(\id_A-t)b$ is a quandle called the {\it Alexander quandle} of $A$ with respect to $t$. Notice that, if $t =-\id_A$, then $a*b=2b-a$. Thus, in this case, the Alexander quandle of $A$ is the Takasaki quandle $T(A)$.
\item The preceding example can be generalised. Let $G$ be a group and $\varphi \in \Aut(G)$. Then the set $G$ equipped with the binary operation $a*b =\varphi(ab^{-1})b$ gives a quandle structure on $G$, called the generalised Alexander quandle of $G$ with respect to $\varphi$.
\item Let $n \ge 2$ and $X= \{a_1, a_2, \dots, a_n\}$ be a set with the binary operation given by $a_i*a_j=a_{n-i+1}$. Then $X$ is a rack which is not a quandle.
\end{itemize}
\bigskip

A quandle or rack $X$ is called {\it trivial} if $x*y=x$ for all $x, y \in X$. Obviously, a trivial rack is a trivial quandle. Unlike groups, a trivial quandle can contain arbitrary number of elements. We denote the $n$-element trivial quandle by $\T_n$.

Notice that, the rack axioms are equivalent to saying that for each $x \in X$, the map $S_x: X \to X$ given by $$S_x(y)=y*x$$ is an automorphism of $X$. Further,  in case of quandles, the axiom $x*x=x$ is equivalent to saying that $S_x$ fixes $x$ for each $x \in X$. Such an automorphism is called an {\it inner automorphism} of $X$, and the group generated by all such automorphisms is denoted by $\Inn(X)$. A quandle $X$ is called {\it involutary} if $S_x^2 = \id_X$ for each $x \in X$. For example, all Takasaki quandles are involutary. 

A {\it loop} is a set $X$ with a binary operation (usually written multiplicatively) and an identity element such that for each $a,b \in X$, there exist unique elements $x,y \in X$ such that both $ax = b$ and $ya = b$. A {\it Moufang loop} is a loop $X$ satisfying the Moufang identity $$(xy)(zx) = (x(yz))x$$ for all $x, y, z \in X$. Unlike groups, Moufang loops need not be associative. In fact, an associative Moufang loop is a group. The {\it core of a Moufang loop} $X$ is the algebraic structure with underlying set $X$ and binary operation $x*y:=yx^{-1}y$. Joyce \cite{Joyce-thesis} showed that the core of a Moufang loop is an involutary quandle. 

A subset $Y$ of a rack $X$ is called a {\it subrack} if $Y$ is a rack with respect to the underlying binary operation. Subquandles are defined analogously. It is easy to see that a subset of an involutary quandle is a subquandle if and only if it is closed under the binary operation.

The group $\Inn(X)$ acts on the quandle $X$ in the obvious way.  A quandle $X$ is said to be {\it connected} if the action of $\Inn(X)$ on $X$ is transitive. It is well-known that the Takasaki quandle of an abelian group $A$ is connected if and only if $A=2A$. In case of finite groups, this is equivalent to saying that $A$ has odd order. Connected quandles are of particular importance since knot quandles are connected. Furthermore, homomorphic images of connected quandles are connected. Therefore, the quandles that appear as homomorphic images of knot quandles (under quandle colorings) are necessarily connected. Thus, classification  of connected quandles is a major research theme, and has attracted a lot of attention \cite {Ehrman, Hulpke, Ishihara, Nelson1, Singh}.
\bigskip

\section{Quandle rings and rack rings}\label{section3}
From this section onwards, for convenience, we denote the multiplication in a quandle or a rack by $(a,b) \mapsto a.b$.

Let $X$ be a quandle and $R$ an associative ring (not necessarily with unity). Let $R[X]$ be the set of all formal finite $R$-linear combinations of elements of $X$, that is,
$$R[X]:=\Big\{ \sum_i\alpha_i x_i~|~\alpha_i \in R,~ x_i \in X \Big\}.$$
Then $R[X]$ is an additive abelian group in the usual way. Define a multiplication in $R[X]$ by setting $$\big(\sum_i\alpha_i x_i\big).\big(\sum_j\beta_j x_j\big):=\sum_{i,j}\alpha_i\beta_j (x_i.x_j).$$
Clearly, the multiplication is distributive with respect to addition from both left and right, and $R[X]$ forms a ring, which we call the {\it quandle ring} of $X$ with coefficients in the ring $R$. Since $X$ is non-associative, unless it is a trivial quandle, it follows that $R[X]$ is a non-associative ring, in general. Analogously, if $X$ is a rack, then we obtain the {\it rack ring} $R[X]$ of $X$ with coefficients in the ring $R$.

We will see that, unlike groups, the quandle ring structure of trivial quandles is quite interesting.

\begin{remark}
Observe that a quandle with a left multiplicative identity has only one element. For, let $e \in X$ be the left identity of $X$. Then $e.x=x$ for all $x \in X$. But, we have $x.x=x$
by  axiom (Q1). Now, by axiom (R1), we must have $e=x$ for all $x \in X$, and hence $X= \{ e \}$. Thus, $R[X]$ is a non-associative ring without unity, unless $X$ is a singleton.
\end{remark}

Analogous to group rings, we define the augmentation map
$$\varepsilon: R[X] \to R$$
 by setting $$\varepsilon \big(\sum_i\alpha_i x_i\big)= \sum_i\alpha_i .$$
Clearly, $\varepsilon$ is a surjective ring homomorphism, and $\Delta_R(X):= \ker(\varepsilon)$ is a two-sided ideal of $R[X]$, called the {\it augmentation ideal} of $R[X]$. Thus, we have
$$R[X]/\Delta_R(X) \cong R$$ as rings. In the case $R = \mathbb{Z}$, we denote the augmentation ideal simply by $\Delta(X)$. The following results is straightforward.

\begin{proposition}
Let $X$ be a rack and $R$ an associative ring. Then $\{x-y~|~x, y \in X \}$ is  a generating set for $\Delta_R(X)$ as an $R$-module. Further, if $x_0 \in X$ is a fixed element, then the set $\big\{x-x_0~|~x \in X \setminus \{ x_0\} \big\}$ is a basis for $\Delta_R(X)$ as an $R$-module.
\end{proposition}

The following is an interesting observation.

\begin{proposition}
Let $X$ be a quandle and $R$ an associative ring. Then $x.y+y.x \equiv x+y~ \mod \Delta_R^2(X)$ for all $x, y \in X$.
\end{proposition}
\begin{proof}
By axiom (Q1), we have $x^2=x$ in $R[X]$ for all $x \in X$. Now,
\begin{eqnarray*}
x & = & (x-y+y)^2\\
& = & (x-y)^2+y^2+(x-y).y+y.(x-y)\\
& \equiv & x.y+y.x-y \mod  \Delta_R^2(X).
\end{eqnarray*}
Thus, $x.y+y.x \equiv x+y \mod \Delta_R^2(X)$ for all $x, y \in X$.
\end{proof}

If $Y$ is a subrack of a rack $X$, then $\Delta_R(Y) \subseteq \Delta_R(X)$. Denote by $\Delta_R(X:Y)$ the two-sided ideal of $R[X]$ generated by $\Delta_R(Y)$. Then
$$
\Delta_R(X: Y) = \Delta_R(Y) + R[X] \Delta_R(Y) + \Delta_R(Y) R[X] + R[X] \Delta_R(Y) R[X].
$$

Given a subrack $Y$ of a rack $X$, it is natural to look for conditions under which $\Delta_R(Y)$ is a two-sided ideal of $R[X]$. For trivial racks, we have the following result.

\begin{proposition}
Let $X$ be a trivial rack, $Y$ a subrack of $X$ and $R$ an associative ring. Then $\Delta_R(Y)$ is a two-sided  ideal of $R[X]$.
\end{proposition}

\begin{proof}
First, notice that, $\Delta_R(Y)$ is generated as an $R$-module by the set $\{y-z~|~y,z \in Y \}$. Then for any $\sum_i\alpha_i x_i \in R[X]$, we have
$$\big(\sum_i\alpha_i x_i\big).(y -z)=\sum_i\alpha_i (x_i .y - x_i . z)=\sum_i\alpha_i (x_i - x_i)= 0 \in \Delta_R(Y),$$
and
$$(y -z). \big(\sum_i\alpha_i x_i\big)=\sum_i\alpha_i (y . x_i -z . x_i)=\sum_i\alpha_i (y -z) \in \Delta_R(Y).$$
Hence $\Delta_R(Y)$ is a two-sided ideal of $R[X]$.
\end{proof}

The next result characterises trivial quandles in terms of their augmentation ideals.

\begin{theorem}\label{deltasqzero}
Let $X$ be a quandle and $R$ an associative ring. Then the quandle $X$ is trivial if and only if $\Delta_R^2(X)=\{0\}$.
\end{theorem}

\begin{proof}
Suppose that $\Delta_R^2(X)=\{0\}$. Let $x_0 \in X$ be an arbitrary but fixed element. Observe that $\Delta_R^2(X)$ is generated as an $R$-module by the set $\big\{(x-x_0).(y-x_0)~|~x, y \in X\setminus \{x_0 \} \big\}$. It follows that $(x-x_0).(y-x_0)=0$ for all $x, y \in X\setminus \{x_0 \}$. In particular, $(x-x_0).(x-x_0)=0$ for all $x \in X\setminus \{x_0 \}$, which yields
$$x-x_0 . x-x . x_0+x_0=0$$
for all $x \in X\setminus \{x_0 \}$. Since it is an expression in $R[X]$, the terms must cancel off with each other. Suppose that $x_0 . x=x$ and $x . x_0=x_0$ for some $x \in X\setminus \{x_0 \}$. Also, we have $x . x=x$ by axiom (Q1). Thus by axiom (R1), we must have $x=x_0$, which is a contradiction. Hence, we must have $x . x_0=x$ and $x_0 . x=x_0$ for all $x \in X\setminus \{x_0 \}$. This means $x_0$ acts trivially on all elements of $X$. Since $x_0$ was an arbitrary element, it follows that the quandle $X$ is trivial.

Conversely, suppose that $X$ is a trivial quandle. Let $y, z, y', z' \in X$. Then
$$(y-z).(y'-z')=y . y'-z . y'-y . z'+z . z'=y-z-y+z=0.$$
By linearity, it follows that $\Delta_R^2(X)=\{0\}$.
\end{proof}

\begin{corollary}
A group $G$ is abelian if and only if $\Delta_R^2\big(\Conj(G)\big)=\{0\}$.
\end{corollary}

\begin{remark}
Obviously, if $X$ is a trivial rack, then $\Delta_R^2(X)=\{0\}$. However, the converse is not true for racks. For example, take $X= \{a_1, \dots, a_n\}$ with the rack structure given by $a_i.a_j=a_{n-i+1}$. Then $X$ is not a trivial rack. On the other hand, for all $1 \le i, j \le n$, we have
$$(a_i-a_1).(a_j-a_1)=a_i.a_j - a_i.a_1 - a_1.a_j + a_1.a_1=a_{n-i+1}-a_{n-i+1}-a_n+a_n=0,$$ and hence  $\Delta_R^2(X)=\{0\}$.
\end{remark}
 \bigskip
%%%%%%%%%%%%%%%%%%%

\section{Relations between subquandles and ideals}\label{section4}
In this section, we investigate relationships between subquandles of the given quandle and ideals of the associated quandle ring.

\subsection{Subquandles associated to ideals}
Let $X$ be a quandle and $R$ an associative ring.  For each $x_0 \in X$ and each two sided ideal $I$ of $R[X]$, we define
$$X_{I, x_0}=\{ x \in X~|~x-x_0 \in I \}.$$

Notice that, if $I= \Delta_R(X)$, then $X_{I, x_0}=X$, and if $I = \{0\}$, then $X_{I,x_0} = \{x_0\}$. In general, we have the following.

\begin{theorem}\label{ideal-to-subquandle}
Let $X$ be a finite quandle and $R$ an associative ring. Then for each $x_0 \in X$ and a two sided ideal $I$  of $R[X]$, the set $X_{I, x_0}$ is a subquandle of $X$. Further, there is a subset $\{ x_1, \ldots, x_m \}$ of $X$  such that $X$ is the disjoint union
$$
X = X_{I,x_1} \sqcup \cdots \sqcup X_{I,x_m}.
$$
\end{theorem}

\begin{proof}
Obviously $x_0 \in X_{I, x_0}$. Further, if $x, y \in X_{I, x_0}$, then $$x.y-x_0= x.y-x_0.y+x_0.y -x_0.x_0=(x-x_0).y+x_0.(y -x_0) \in I.$$
This implies that $x. y \in X_{I, x_0}$. Further, the inner automorphism $S_y$ restricts to a map $X_{I, x_0}  \to X_{I, x_0}$. Since $S_y$ is injective and $X$ is finite, it follows that  $S_y\big(X_{I, x_0} \big)= X_{I, x_0}$. Hence, there exists a unique element $z \in X_{I, x_0} $ such that $x=z.y$, thereby proving that $X_{I, x_0}$ is a subquandle of $X$.

For the second assertion, it is sufficient to prove that if $x_0, y_0 \in X$, then the subquandles $X_{I, x_0}$ and $X_{I, y_0}$ intersect if and only if they are equal. Let $z \in X_{I, x_0} \cap X_{I, y_0}$. Then we have $x_0 - y_0 \in I$, which further implies that $x_0 \in X_{I, y_0}$. Now, if $x \in X_{I, x_0}$, then $x-y_0=x-x_0+x_0-y_0 \in I$, which implies that $X_{I, x_0} \subseteq X_{I, y_0}$. By interchanging roles of $x_0$ and $y_0$, we get $X_{I, x_0} = X_{I, y_0}$.
\end{proof}

\begin{remark}
If $X$ is an involutary quandle (not necessarily finite), then $X_{I, x_0}$ is always a subquandle of $X$, being closed under the quandle multiplication. In particular, this holds for trivial quandles.
\end{remark}

Given a two sided ideal $I$ of $R[X]$ and elements $x_0, y_0 \in X$, it is natural to ask whether there is any relation between the subquandles $X_{I, x_0}$ and $X_{I, y_0}$. We answer this question in the following result.

\begin{theorem}\label{id-subq}
Let $X$ be a finite involutary quandle and $R$ an associative ring.  If $I$ is a two-sided ideal of $R[X]$ and $x_0, y_0 \in X$ are in the same orbit under action of $\Inn(X)$, then $X_{I, x_0} \cong X_{I, y_0}$.
\end{theorem}

\begin{proof}
We first claim that if $x_0, y_0 \in X$, then the map $f_{y_0}: X_{I, x_0} \to X_{I, x_0.y_0}$ given by $f_{y_0}(x)=x.y_0$ is an injective quandle homomorphism. For, if $x \in X_{I, x_0}$, then $x-x_0 \in I$. Consequently, $x.y_0-x_0.y_0= (x-x_0).y_0 \in I$, which further implies $f_{y_0}(x)=x.y_0 \in X_{I, x_0.y_0}$. The claim now follows by observing that $f_{y_0}$ is simply the restriction of the inner automorphism $S_{y_0}$ on the subquandle $X_{I, x_0}$.

Now, suppose that $x_0, y_0 \in X$ such that there exists $f \in \Inn(X)$ with $f(x_0)=y_0$. Since $X$ is involutary, $S_x=S_x^{-1}$ for all $x \in X$. Therefore, by definition $f= S_{x_k}  \cdots  S_{x_1}$ for some $x_i \in X$, and
$$y_0= S_{x_k}  \cdots   S_{x_1}(x_0)= (\cdots ((x_0. x_1) . x_2) \cdots ) . x_k.$$
By the claim above, there is a sequence of embeddings of subquandles of $X$
$$ X_{I, x_0} \hookrightarrow X_{I, x_0. x_1} \hookrightarrow X_{I, (x_0. x_1). x_2} \hookrightarrow \cdots \hookrightarrow  X_{I, y_0}.$$
Thus, we obtain an embedding of $X_{I, x_0}$ into $X_{I, y_0}$. Writing $x_0= S_{x_1}  \cdots   S_{x_k}(y_0)$, we obtain an embedding of $X_{I, y_0}$ into $X_{I, x_0}$. Since $X$ is finite, it follows that $X_{I, x_0} \cong X_{I, y_0}$.
\end{proof}

The following is an immediate consequence.

\begin{corollary}
Let $X$ be a finite connected involutary quandle and $R$ an associative ring. Then $X_{I, x_0} \cong X_{I, y_0}$ for any ideal $I$ of $R[X]$ and $x_0, y_0 \in X$.
\end{corollary}

For example, trivial quandles and dihedral quandles $\R_n$ of odd order are connected  and involutary.  The example below shows that we cannot write $X_{I, x_0} = X_{I, y_0}$ instead $X_{I, x_0} \cong X_{I, y_0}$ in Theorem \ref{id-subq}.

\begin{example}
Let $T_2 = \{ x, y \}$ be the two element trivial quandle and $I$ the two-sided ideal in $\mathbb{Z} [T_2]$, generated by the element $u = 2 x + 2y$. Since
$$
ux = u,~~uy = u,~~xu = 4 x,~~yu = 4 y,
$$
it is not difficult to see that
$$
I = \big\{ 2 \alpha x + (2 \alpha + 4 \beta) y~|~\alpha, \beta \in \mathbb{Z} \big\}.
$$
Obviously $x \in (T_2)_{I,x}$. Now, if $y \in (T_2)_{I,x}$, then $y-x \in I$, i.e. $y-x = 2 \alpha x + (2 \alpha + 4 \beta) y$ for some integers $\alpha$ and $\beta$, which is not possible. This implies $(T_2)_{I,x} = \{ x \}$. Similarly $(T_2)_{I,y} = \{ y \}$, and therefore $(T_2)_{I,x} \not= (T_2)_{I,y}$.
\end{example}

Notice that, in the preceding example, we have $I = 2\mathbb{Z} \cdot I_1$, where $2\mathbb{Z}$ is an ideal of $\mathbb{Z}$ and
$$
I_1 = \{\alpha x + (\alpha + 2 \beta) y~|~\alpha, \beta \in \mathbb{Z} \}
$$
is a proper ideal of $\mathbb{Z}[T_2]$ since $y \not\in I_1$.

Let $X$ be a quandle and $R$ an associative ring. We say that an ideal $I$ of the rack ring $R[X]$ is $R$-{\it prime} if $I$ is not a product $I = I_0 \cdot I_1$, where $I_0$ is a proper ideal of $R$ and $I_1$ is a proper ideal of $R[X]$. In view of this definition, it is tempting to ask whether the condition of $I$ being $R$-prime in Theorem \ref{id-subq} imply that $X_{I, x_0} = X_{I, y_0}$? The following example shows that it is not true in general.

\begin{example}
Let $X = R_4 = \{ a_0, a_1, a_2, a_3 \}$ be the dihedral quandle. We know that $X$ is not connected since  the elements $a_0$ and $a_1$ lying in different orbits. In the quandle ring $\mathbb{Z}[R_4]$ take the two-sided ideal $I$ with the linear basis $\{ a_0 + a_1 - a_2 - a_3, 2(a_2 - a_0) \}$ (in fact $I = \Delta^2(R_4)$). It is not difficult to check  that $X_{I, x_0} = \{ x_0 \}$ and $X_{I, y_0} = \{ y_0 \}$, i.~e. $X_{I, x_0} \not= X_{I, y_0}$. Moreover, $X$ is the following disjoint union of its trivial subquandles:
$$
X = X_{I, a_0} \sqcup X_{I, a_1} \sqcup X_{I, a_2} \sqcup X_{I, a_3}.
$$
\end{example}

\begin{remark}
It is worth noting that the results of the preceding discussion holds for racks as well. It would be interesting to explore whether Theorem \ref{id-subq} holds if $X$ is not involutary or  if $X$ is  involutary but $x_0$ and $y_0$ lie in different orbits under the action of $\Inn(X)$.
\end{remark}
\bigskip
%%%%%%%%%%%%%%%%%%%%%%%
%%%%%%%%%%%%%%%%%%%%%%%

\subsection{Ideals associated to subquandles}
Next, we proceed in the reverse direction of associating an ideal of $R[X]$ to a subquandle of $X$. Let $f:X \to Z$ be a quandle homomorphism. Consider the equivalence relation $\sim$ on $X$ given by $x_1\sim x_2$ if $f(x_1)=f(x_2)$. Let $X/_\sim$ be the set of equivalence classes, where equivalence class of an element $x$ is denoted by $$X_x:= \{x' \in X~|~f(x')=f(x) \}.$$

\begin{proposition}
$X_x$ is a subquandle of $X$ for each $x \in X$.
\end{proposition}

\begin{proof}
Obviously $X_x$ is non-empty since $x \in X_x$. Only the axiom (R1) needs to be checked. Let $x_1, x_2 \in X_x$. 
Then $f(x_1 . x_2)=f(x_1) .f(x_2)=f(x).f(x)=f(x)$, and hence $x_1.x_2 \in X_x$. Further, if $x_3 \in X$ is the unique element such that $x_3.x_1=x_2$, then applying $f$ yields $f(x_3) .f(x)=f(x)$. This together with the axiom (Q1) imply that $f(x_3)=f(x)$, and hence $x_3 \in X_x$.
\end{proof}

The following is a sort of first isomorphism theorem for quandles.

\begin{theorem}\label{quotient-quandle}
The binary operation given by $X_{x_1} \circ X_{x_2}= X_{x_1.x_2}$ gives a quandle structure on $X/_\sim$. Further, if $f:X \to Z$ is a surjective quandle homomorphism, then $X/_\sim \cong Z$ as quandles.
\end{theorem}

\begin{proof}
The operation $X_{x_1} \circ X_{x_2}= X_{x_1.x_2}$ is clearly well-defined. We only need to check the axiom (R1). Let $X_{x_1}, X_{x_2} \in X/_\sim$.  If $x_3 \in X$ is the unique element such that $x_3.x_1=x_2$, then $f(x_3).f(x_1)=f(x_2)$ and $X_{x_3} \circ X_{x_1}=X_{x_2}$. Suppose that there exists another element $X_{x_3'} \in X/_\sim$ such that $X_{x_3'} \circ X_{x_1}=X_{x_2}$, then  $f(x_3').f(x_1)=f(x_2)$. By the axiom (Q1), we must have $f(x_3)=f(x_3')$, and hence $X_{x_3} =X_{x_3'}$.

Suppose that $f:X \to Z$ is surjective. By definition of the equivalence relation, there is a well-defined bijective map $\bar{f}: X/_\sim \to Z$ given by $\bar{f}(X_x)=f(x)$. If $X_{x_1}, X_{x_2} \in X/_\sim$, then
$$\bar{f}(X_{x_1}\circ X_{x_2})=\bar{f}(X_{x_1. x_2})=f(x_1.x_2)=f(x_1).f(x_2)=\bar{f}(X_{x_1}). \bar{f}(X_{x_2}),$$
and hence $\bar{f}$ is an isomorphism of quandles.
\end{proof}

We know that a subgroup of a group is normal if and only if it is the kernel of some group homomorphism. In a similar way, we say that a subquandle $Y$ of a quandle $X$ is {\it normal} if $Y=X_{x_0}$ for some $x_0 \in X$ and some quandle homomorphism $f:X \to Z$. In this case, we say that $Y$ is {\it normal based at} $x_0$.

A {\it pointed quandle}, denoted $(X, x_0)$, is a quandle $X$ together with a fixed base point $x_0$. Let $f:(X, x_0) \to (Z, z_0)$ be a homomorphism of pointed quandles, and $Y=X_{x_0}$ a normal subquandle based at $x_0$. In this situation, we consider $Y$ as the base point of $X/_\sim$, and denote $X/_\sim$ by $X/Y$. Then the natural map $x \mapsto X_x$ is a surjective homomorphism of pointed quandles $$(X, x_0) \to (X/Y, X_{x_0}).$$ This further extends to a surjective ring homomorphism, say, $$\pi: R[X] \to R[X/Y]$$ with $\ker(\pi)$ being a two sided ideal of $R[X]$.

Let $(X, x_0)$ be a pointed quandle, $\mathcal{I}$ the set of two sided ideals of $R[X]$ and $\mathcal{S}$ the set of normal subquandles of $X$ based $x_0$. Then there exist maps $\Phi: \mathcal{I} \to \mathcal{S}$ given by
$$\Phi(I)=X_{I,x_0}$$
and $\Psi: \mathcal{S} \to \mathcal{I}$ given by
$$\Psi(Y)=\ker(\pi).$$
With this set up, we have the following.

\begin{theorem}\label{dictionary}
Let $(X, x_0)$ be a pointed quandle and $R$ an associative ring. Then $\Phi \Psi= \id_{\mathcal{S}}$ and $\Psi \Phi \neq \id_{\mathcal{I}}$.
\end{theorem}

\begin{proof}
Let $Y=X_{x_0}$ be a normal subquandle of $X$, that is,  $Y=\{x \in X~|~f(x)=f(x_0) \}$. Then
\begin{eqnarray*}
\Phi \Psi(Y) & = & \Phi \big( \ker(\pi) \big)\\
&=& \big\{x \in X~|~x-x_0 \in  \ker(\pi) \big\}\\
&=& \big\{x \in X~|~X_x=X_{x_0}\big\}\\
&=& \big\{x \in X~|~f(x)=f(x_0)\big\}\\
&=& Y.
\end{eqnarray*}
Obviously, $\Psi \Phi \neq \id_{\mathcal{I}}$, since $\Psi \Phi \big(R[X] \big)=\Psi (X)= \Delta_R(X) \neq R[X]$.
\end{proof}
\bigskip
%%%%%%%%%%%%%%%%%%%%%%%%%%%%%%%%%%%%%%

\section{Extended rack ring and units}\label{section5}
In this section, we assume that $R$ is an associative ring with unity 1. Let  $X$ be a rack. Since $R[X]$ is a ring without unity, it is desirable to embed $R[X]$ into a ring with unity. The ring
$$
R^\circ[X]=R[X] \oplus Re,
$$
where $e$ is a symbol satisfying $e\big(\sum_i\alpha_i x_i\big)= \sum_i\alpha_i x_i= \big(\sum_i\alpha_i x_i\big)e$, is called the {\it extended rack ring} of $X$. For convenience, we denote the unity $1e$ of $R^\circ[X]$ by $e$. We extend the augmentation map $\varepsilon: R^\circ[X] \to R$ to obtain the extended augmentation ideal
$$\Delta_{R^\circ}(X):= \ker(\varepsilon: R^\circ[X] \to R).$$
 In the case $R = \mathbb{Z}$, we simply denote it by $\Delta_{\circ}(X)$. As before, it is easy to see that the set $\{x-e~|~ x \in X  \}$ is a basis for $\Delta_{R^\circ}(X)$ as an $R$-module.

\begin{proposition}
If $X$ is a rack and $x_0 \in X$ a fixed element, then $\Delta_{R^\circ}(X) =\Delta_R(X) +R(e-x_0)$.
\end{proposition}

\begin{proof}
Let $u \in \Delta_{R^\circ}(X)$. Then $u= \sum_i\alpha_i x_i + \beta e$, where $\sum_i\alpha_i + \beta=0$ and $\alpha_i, \beta \in R$. Thus, we can rewrite $u=\sum_i\alpha_i (x_i-x_0) +\beta (e-x_0)$, where $\sum_i\alpha_i (x_i-x_0)  \in \Delta_R(X)$. Thus every element  $u \in \Delta_{R^\circ}(X)$ can be written as $u= a+ \beta(e-x_0)$, where $a \in \Delta_R(X)$ and $\beta \in R$.
\end{proof}

\begin{proposition}\label{sq-delta-o}
 If $X$ is a quandle, then $\Delta_{R^\circ}^2(X) = \Delta_{R^\circ}(X)$.
\end{proposition}

\begin{proof}
Obviously we have $\Delta_{R^\circ}^2(X)\subseteq \Delta_{R^\circ}(X)$. To prove the opposite inclusion,  notice that,
$\Delta_{R^\circ}^2(X)$ is generated by $\{(x_i - e). (x_j - e)~|~x_i, x_j \in X\}$. Taking $x_j = x_i$, we have
$$
(x_i - e). (x_i - e) = -(x_i - e) \in \Delta_{R^\circ}(X).
$$
Hence $\Delta_{R^\circ}(X) \subseteq \Delta_{R^\circ}^2(X)$,  and proposition is proved.
\end{proof}

\begin{remark}
If $X$ is a rack with $\Delta_{R^\circ}^2(X) = \Delta_{R^\circ}(X)$, then it is not necessarily a quandle. For example, if $X= \{a, b \}$ with the rack structure $a.a=a.b=b$ and $b.a=b.b=a$, then  $\Delta_{R^\circ}^2(X) = \Delta_{R^\circ}(X)$, but $X$ is not a quandle.
\end{remark}
\bigskip

\subsection{Units in extended rack rings}\label{section5.1}
Let $X$ be a rack and $R$ an associative ring with unity 1. Though the ring $R^\circ[X]$ has unity, it is non-associative, in general. Therefore, a natural problem is to determine maximal multiplicative subgroups of $R^\circ[X]$.

Let $\mathcal{U}(R^\circ[X])$ denote a maximal multiplicative subgroup of the ring $R^\circ[X]$. Notice that, $\varepsilon: R^\circ[X] \to R$ maps $\mathcal{U}(R^{\circ}[X])$ onto $R^*$, the group of units of $R$. Let
$$
\mathcal{U}_1(R^{\circ}[X]) := \big\{ r \in \mathcal{U}(R^{\circ}[X])~|~\varepsilon(r) = 1 \big\},
$$
be the subgroup of normalized units. Then $\mathcal{U}(R^{\circ}[X]) = R^* ~\mathcal{U}_1(R^{\circ}[X])$, and one only need to compute the group of normalized units. Define
$$
\mathcal{V}(R^{\circ}[X]) := \big\{ e+a \in \mathcal{U}(R^{\circ}[X]) ~|~a \in R[X] \big\}.
$$
Then it is not difficult to see that $\mathcal{V}(R^{\circ}[X])$ is a normal subgroup of $\mathcal{U}(R^{\circ}[X])$ and  $\mathcal{U}(R^{\circ}[X]) = R^* ~ \mathcal{V}(R^{\circ}[X])$. To understand  $\mathcal{V}(R^{\circ}[X])$ further, we define
$$\mathcal{V}_1 (R^{\circ}[X]) := \mathcal{U}_1(R^{\circ}[X]) \cap \mathcal{V}(R^{\circ}[X]).$$

\begin{proposition}
Let $X$ be a rack and $R$ an associative ring with unity. Then  $ \mathcal{V}_1 (R^{\circ}[X])= \big\{ e+a \in \mathcal{U}(R^{\circ}[X]) ~|~a \in \Delta_R(X) \big\}$ and
is a normal subgroup of $\mathcal{U}(R^{\circ}[X])$.
\end{proposition}

\begin{proof}
Let $r = e + a \in \mathcal{V}(R^{\circ}[X])$. Then $\varepsilon(r) = 1 + \varepsilon(a)$. On the other side, if $r  \in \mathcal{U}_1(R^{\circ}[X])$, then  $\varepsilon(r)=1$. Hence $\varepsilon(a) = 0$, i. e. $a \in \Delta(R[X])$, and the first assertion is proved.

Let $u \in \mathcal{U}(R^{\circ}[X])$ and $r = e + a \in \mathcal{V}_1(R^{\circ}[X])$. Then $u^{-1}.r.u = e + u^{-1}. a .u$
and $\varepsilon(u^{-1}.r.u) = 1 + \varepsilon(u^{-1} .a .u)$, where $\varepsilon(u^{-1}. a. u)  = \varepsilon(a)=0$. Hence $u^{-1}.r.u \in \mathcal{V}_1 (R^{\circ}[X])$ proving the second assertion.
\end{proof}

Let $X$ be a rack and $x_0 \in X$ a fixed element. Define the set
$$
\mathcal{V}_2(R^{\circ}[X]) = \big\{ e + (\lambda - 1) x_0~|~\lambda \in R^* \}.
$$

\begin{proposition}\label{V_2-subgroup}
Let $X$ be a rack, $x_0 \in X$ a fixed element and $R$ an associative ring with unity. Then $\mathcal{V}_2(R^{\circ}[X])$ is a subgroup of $\mathcal{V}(R^{\circ}[X])$ and is isomorphic to $R^*$.
\end{proposition}

\begin{proof}
Notice that, if $r = e + (\lambda - 1) x_0 \in \mathcal{V}_2(R^{\circ}[X])$, then $r^{-1} = e + (\lambda^{-1} - 1) x_0 \in \mathcal{V}_2(R^{\circ}[X])$, and hence $r \in \mathcal{V}(R^{\circ}[X])$. Further, if $r = e + (\lambda - 1) x_0$ and  $s= e + (\mu - 1) x_0$, then $rs= e + (\lambda\mu - 1) x_0$. Hence $\mathcal{V}_2(R^{\circ}[X])$ is a subgroup of $\mathcal{V}(R^{\circ}[X])$. Finally, the map $\lambda \mapsto e + (\lambda - 1) x_0$ gives an isomorphism of $R^*$ onto $\mathcal{V}_2(R^{\circ}[X])$.
\end{proof}

\begin{theorem}\label{split-short-seq}
Let $X$ be a rack, $x_0 \in X$ a fixed element and $R$ an associative ring with unity. Then there exists a split exact sequence
$$
1 \longrightarrow \mathcal{V}_1(R^{\circ}[X]) \longrightarrow \mathcal{V}(R^{\circ}[X]) \longrightarrow \mathcal{V}_2(R^{\circ}[X]) \longrightarrow 1.
$$
\end{theorem}

\begin{proof}
Let $v = e + a \in \mathcal{V}(R^{\circ}[X])$. Then $\varepsilon(v) = 1 + \varepsilon(a) \in R^*$. Since $a \in R[X]$, we have $a = \sum_{i=0}^n \alpha_i x_i$ for  $\alpha_i \in R$, $ x_i \in X$ and $\varepsilon(a) = \sum_{i=0}^n \alpha_i$. Rewrite $a$ in the form
$$
a = \sum_{i=1}^n \alpha_i (x_i - x_0) +
\left( \sum_{i=0}^n \alpha_i \right) x_0 = \sum_{i=1}^n \alpha_i (x_i - x_0) + (\varepsilon(v) - 1) x_0,
$$
where
$$
\sum_{i=1}^n \alpha_i (x_i - x_0) \in \Delta_R(X)~\textrm{and}~\varepsilon(v) \in R^*.
$$
Hence, we can write
$$
v = e + a_0 + (\varepsilon(v) - 1) x_0,~~~a_0 \in \Delta_R(X).
$$
Now define a map
$$
\varphi : \mathcal{V}(R^{\circ}[X]) \longrightarrow \mathcal{V}_2(R^{\circ}[X])
$$
by the rule $\varphi(e + a_0 + (\varepsilon(v) - 1) x_0) = e +  (\varepsilon(v) - 1) x_0$. If $v = e + a_0 + (\varepsilon(v) - 1) x_0$ and $u = e + b_0 + (\varepsilon(u) - 1) x_0$ are elements of $\mathcal{V}(R^{\circ}[X])$, then
$$
v.u = e + a_0 + b_0 + (\varepsilon(v) - 1) x_0.b_0 + (\varepsilon(u) - 1) a_0.x_0 + a_0.b_0 + (\varepsilon(v) \varepsilon(u) - 1) x_0
$$
and $\varphi(v.u)= e +  (\varepsilon(v) \varepsilon(u) - 1) x_0=\varphi(v).\varphi(u)$. Clearly, $\varphi$ is surjective and $\ker(\varphi) =\mathcal{V}_1(R^{\circ}[X])$. Thus we obtain the short exact sequence
$$1 \longrightarrow \mathcal{V}_1(R^{\circ}[X]) \longrightarrow \mathcal{V}(R^{\circ}[X]) \longrightarrow \mathcal{V}_2(R^{\circ}[X]) \longrightarrow 1,$$
which splits by Proposition \ref{V_2-subgroup}.
\end{proof}

As  a consequence, we have the following.

\begin{corollary}
Let $X$ be a rack, $x_0 \in X$ and $R$ an associative ring with unity.  Then an arbitrary element $u = e + a_0 + (\lambda - 1)x_0 \in \mathcal{V}(R^{\circ}[X])$, $\lambda \in R*$,  is the product $u = u_1 u_2$, where $u_1 = e + a_0\big(e+(\lambda^{-1} - 1) x_0\big) \in \mathcal{V}_1(R^{\circ}[X])$ and $u_2 = e + (\lambda - 1) x_0 \in \mathcal{V}_2(R^{\circ}[X])$.
\end{corollary}

Since $\mathcal{V}_2(R^{\circ}[X]) \cong R^*$, it follows from Theorem \ref{split-short-seq} that if we know the structure of $\mathcal{V}_1(R^{\circ}[X]$, then we can determine the structure of $\mathcal{V}(R^{\circ}[X]$. The preceding discussion can be summarised in the following digram.

$$
\xymatrix{
& \mathcal{U}(R^{\circ}[X]) &  \\
\mathcal{U} _1(R^{\circ}[X])   \ar[ru] &  &  \ar[lu] \mathcal{V}(R^{\circ}[X])= \mathcal{V}_1(R^{\circ}[X]) \oplus \mathcal{V}_2(R^{\circ}[X])  \\
&  \mathcal{V}_1(R^{\circ}[X]) \ar[ru]  \ar[lu]&  }
$$

\bigskip
%%%%%%%%%%%%%%%%%%%%%%%%%%%%%%%

\subsection{Units in extended rack rings of trivial racks}\label{section5.2}
 Notice that, if $\T$ is a trivial rack and $R$ an associative ring with unity,  then the ring $R^\circ[\T]$ is associative with  unity, and hence $\mathcal{U}(R^\circ[\T])$ is precisely the group of units of $R^\circ[\T]$.  As observed in the preceding section, the main problem in determining $\mathcal{U}(R^{\circ}[T])$ is the description of $\mathcal{V}_1(R^{\circ}[T])$.

\begin{proposition}
Let $\T$ be a trivial rack and $R$ an associative ring with unity. Then
$$
\mathcal{V}_1(R^{\circ}[\T])  = \big\{ e+a  ~|~a \in \Delta_R(\T) \big\}
$$
is an abelian subgroup of $\mathcal{U}_1(R^{\circ}[\T])$. Further, $\mathcal{V}_1(R^{\circ}[\T]) \cong \Delta_R(\T)$.
\end{proposition}

\begin{proof}
Let $a, b \in \Delta_R(\T)$. Then $(e+a).(e+b)=e+a+b+a.b=e+a+b=(e+b).(e+a)$, since $ \Delta_R^2(\T)=\{0\}$ by Theorem \ref{deltasqzero}. This implies that  $\mathcal{V}_1(R^{\circ}[\T])$   is closed under multiplication and is abelian. Further, $(e+a).(e-a)=e+a-a+0=e$ shows that  $\mathcal{V}_1(R^{\circ}[\T])$   is an abelian subgroup of $\mathcal{U}_1(R^{\circ}[X])$. Finally, the map $e+a \mapsto a$ gives the desired isomorphism of $\mathcal{V}_1(R^{\circ}[\T])$ onto the additive group $\Delta_R(\T)$.
\end{proof}

More generally, we prove the following.

\begin{theorem}\label{unit-trivial-rack}
Let $\T$ be a trivial rack, $x_0 \in \T$ and $R$ an associative ring with unity. Then $\mathcal{U}_1(R^{\circ}[\T])=\big\{e+a +\alpha(x_0- e)~|~a \in \Delta_R(\T)~\textrm{and}~\alpha-1 \in R^* \big\}.$
\end{theorem}

\begin{proof}
Let $u \in \mathcal{U}_1(R^\circ[\T])$. Then $\varepsilon(u)= 1$, and hence $-e+u \in \Delta_R^\circ(\T)$. This implies that $u=e+w$ for some $w=\sum_i\alpha_i (x_i- e) \in \Delta_R^\circ(\T)$. Now let $x_0 \in \T$ be a fixed element. Then we can write $$u=e+\sum_i\alpha_i (x_i- e)=e+\sum_{i \neq 0} \alpha_i (x_i- x_0)+ \sum_i \alpha_i (x_0- e),$$
where $\sum_{i \neq 0} \alpha_i (x_i- x_0) \in \Delta_R(\T)$. Thus we have written $u=e+a + \alpha (x_0- e)$ for some $a \in  \Delta_R(\T)$ and $\alpha \in R$. Since $u$ is a unit, there exists an element $\bar{u}=e+b+ \beta  (x_0- e)$ such that $u .\bar{u}=1$. Then equating the two sides gives
$$a+(1-\alpha)b +(\alpha+\beta-\alpha \beta)(x_0-e)=0,$$
which in turn gives $a= (\alpha-1)b$ and $\alpha+\beta=\alpha \beta$. Clearly, $\alpha \neq 1$, and hence the only solution is $\beta=\frac{\alpha}{\alpha-1}$, which is possible  iff that $\alpha-1 \in R^*$.

Conversely, consider an element $u=e+a +\alpha(x_0- e)$, where $a \in \Delta_R(\T)$ and $\alpha-1 \in R^*$. Taking $\bar{u}=  e+\frac{1}{\alpha-1}a +\frac{\alpha}{\alpha-1}(x_0- e)$, we easily see that $u .\bar{u}=\bar{u} .u=1$. Hence $\bar{u}$ is the inverse of $u$, and the proof is complete.
\end{proof}

As a consequence, we have the following  for integral coefficients.

\begin{corollary}
Let $\T$ be a trivial rack, $x_0 \in X$ and $R$ an associative ring with unity. Then the following statements hold:
\begin{enumerate}
\item  $\mathcal{U}(\mathbb{Z}^{\circ}[\T])=\pm \big\{e+a +\alpha(x_0- e)~|~a \in \Delta_R(\T)~\textrm{and}~\alpha=0,2 \big\}$.
\item If  $\T_1 = \{x_0\}$,  then $\mathcal{U}(\mathbb{Z}^{\circ}[\T_1]) = \{ \pm e, 2 x - e, -2 x + e \} \cong \mathbb{Z}_2 \oplus \mathbb{Z}_2$.
\end{enumerate}
\end{corollary}

Next we consider nilpotency of $\mathcal{V}(R^{\circ}[\T])$. Obviously, $\mathcal{V}(R^{\circ}[\T_1])$ is nilpotent. In the general case, we prove the following.

\begin{proposition}
Let $R$ be an associative ring with unity such that $|R^*| > 1$. If $\T$ is a trivial rack with more than one element, then $\mathcal{V}(R^{\circ}[\T])$ is not nilpotent.
\end{proposition}

\begin{proof}
It is enough to prove that for the two element trivial rack $\T_2 = \{ x, y \}$ the group $\mathcal{V}(R^{\circ}[\T_2])$ is not nilpotent. Let
$$
v = e + (y-x) + (\varepsilon(v) - 1) x~\textrm{and}~u = e + (y-x) + (\varepsilon(u) - 1) x
$$
be two elements of $\mathcal{V}(R^{\circ}[\T_2])$ with   $\varepsilon(v) \not= \varepsilon(u)$ and $\varepsilon(u) \not= 1$. Then
$$
v^{-1} = e - \varepsilon(v)^{-1} (y-x) + (\varepsilon(v)^{-1} - 1) x,~~~u^{-1} = e - \varepsilon(u)^{-1} (y-x) + (\varepsilon(u)^{-1} - 1) x.
$$

Now consider the following sequence of elements
$$
w_1 = [v, u] ~\textrm{and}~w_n = [w_{n-1},  u]~\mbox{for}~n >1.
$$
We have
$$
w_1 = [v, u] = e + (\varepsilon(u) - \varepsilon(v)) (y-x).
$$
 Using induction on $n$, we get
$$
w_n = e + (\varepsilon(u) - \varepsilon(v)) (\varepsilon(u) - 1)^{n-1} (y-x).
$$
Since $\varepsilon(v) \not= \varepsilon(u)$ and $\varepsilon(u) \not= 1$, the elements $w_n$ are not trivial, and hence $\mathcal{V}(R^{\circ}[\T])$ is not nilpotent.
\end{proof}
\bigskip
%%%%%%%%%%%%%%%%%%%%%%%%%%%%%%

An important ingredient in the study of  units of group rings are the central units.  We define the {\it center} $\Z(R[X])$ of the rack ring $R[X]$ as
$$
\Z(R[X]) := \big\{ u \in R[X] ~|~u.v = v.u~\mbox{for all}~v \in R[X] \big\}.
$$
Notice that, $\Z(R[X])$ is clearly an additive subgroup of the rack ring $R[X]$. Since $R[X]$, in general, is non-associative, it follows that $\Z(R[X])$ need not be a subring of  $R[X]$. Of course, if $X$ is a trivial rack, then $\Z(R[X])$ is a subring of $R[X]$. However, even in this case, $\Z(R[X])$ need not be an ideal of $R[X]$.

Recall that, a rack is called {\it latin} if the left multiplication by each element is a permutation of the rack. For example, odd order dihedral quandles are latin. Notice that, a latin rack is always connected.

\begin{proposition}
Let $X = \{ x_1, x_2, \ldots, x_n \}$ be a finite latin rack. Then the following hold:
\begin{enumerate}
\item The element $w = x_1 + x_2 + \cdots + x_n$ lies in the center $\Z(R[X])$.
\item If $R$ is a field with characteristic not dividing $n$, then $\frac{1}{n}w$ is idempotent.
\item If $R$ is a field with characteristic not dividing $n$, then any element of the form $e + \alpha w$ with $\alpha \neq -1/n$ is a central unit with inverse $e  -\frac{\alpha}{1+n \alpha} w$.
\end{enumerate}
\end{proposition}

\begin{proof}
Notice that, for each $x_i$, we have
$$
w . x_i = S_{x_i}(x_1)+S_{x_i}(x_2)+ \cdots + S_{x_i}(x_n) = w,
$$
since $S_{x_i}$ is an automorphism of $X$ and acts as a permutation. On the other hand, since the rack $X$ is latin, we have $x_i . w = w$ proving (1).

If the characteristic of $R$ does not divide $n$, then a direct computation yields (2).

For (3), suppose that $e + \alpha w$ has inverse $e +  \beta w$. Then, using (2), we have
$$
(e + \alpha w). (e + \beta w) = e + (\alpha + \beta + n \alpha \beta) w=e.
$$
Consequently, $\alpha + \beta + n\alpha \beta = 0$, and hence $\beta = -\alpha/(1+n \alpha)$. By (1), $e + \alpha w$ is clearly central.
\end{proof}
\bigskip
%%%%%%%%%%%%%%%%%%%%%%%%%%%%%%%%%%%%%%%%%%
%%%%%%%%%%%%%%%%%%%%%%%%%%%%%%%%%%%%%%%%%%

\section{Associated graded ring of a rack ring}\label{section6}
Let $X$ be a rack and $R$ an associative ring. Consider the direct sum
$$
\mathcal{X}_R(X) := \sum_{i\geq 0} \Delta^i_R(X) / \Delta^{i+1}_R(X)
$$
of $R$-modules $\Delta^i_R(X) / \Delta^{i+1}_R(X)$. We regard $\mathcal{X}_R(X)$ as a graded $R$-module with the convention that the elements of $\Delta^i_R(X) / \Delta^{i+1}_R(X)$ are homogeneous of degree $i$. Define multiplication in $\mathcal{X}_R(X)$ as follows. For $u_i \in \Delta^i_R(X)$, let $\underline{u}_i = u_i + \Delta^{i+1}_R(X)$. Then, for $\underline{u}_i \in \Delta^i_R(X) / \Delta^{i+1}_R(X)$, $\underline{u}_j \in \Delta^j_R(X) / \Delta^{j+1}_R(X)$, we define $\underline{u}_i \cdot \underline{u}_j = \underline{u_i u_j}$. The product of two arbitrary elements of $\mathcal{X}_R(X)$ is defined by extending the above product by linearity, that is, if $\underline{u} = \sum \underline{u}_i$ and $\underline{v} = \sum \underline{v}_j$ are decompositions of $\underline{u}, \underline{v} \in \mathcal{X}_R(X)$ into homogeneous components, then
$$
\underline{u} \cdot \underline{v} = \sum_{i,j} \underline{u}_i \cdot \underline{v}_j.
$$
With this multiplication, $\mathcal{X}_R(X)$ becomes a graded ring, and we call it the {\it associated graded ring of} $R[X]$.

For trivial quandles we have a full description of its associated graded ring.

\begin{proposition}
If $T$ is a trivial quandle and $x_0 \in T$, then $\mathcal{X}_R(\T) = R x_0\oplus \Delta_R(\T)$.
\end{proposition}

\begin{proof}
It follows from the definition of $\mathcal{X}_R(\T)$ and Theorem \ref{deltasqzero}.
\end{proof}

For studying $\mathcal{X}_R(X)$ we need to understand the quotients $\Delta^i_R(X) / \Delta^{i+1}_R(X)$. In the case of groups, we
have the following result (see \cite[p. 122]{Passi}): Let $G$ be a finite group and $Q_n(G) = \Delta^n_{\mathbb{Z}} (G) / \Delta^{n+1}_{\mathbb{Z}} (G)$, then there exist integers $n_0$ and $\pi$ such that
$$
Q_n(G) \cong Q_{n+\pi}(G) ~\mbox{for all}~n \geq n_0.
$$

In what follows,  we compute  powers of the integral augmentation ideals of the dihedral quandles $\R_n$ for some small values of $n$. Observe that $\R_2 = \T_2$, the trivial quandle with 2 elements. Consider the integral quandle ring of the dihedral quandle $\R_3= \{a_0, a_1, a_2 \}$. Set $e_1:= a_1-a_0$ and $e_2:=a_2-a_0$.
Then  $\Delta(\R_3)= \langle e_1, e_2 \rangle $. To determine  $\Delta^2(\R_3)$, we compute the products $e_i e_j$ and put them in the following table.

\begin{center}
\begin{tabular}{|c||c|c|}
  \hline
 $\cdot$ & $e_1$ & $e_2$  \\
   \hline
     \hline
$e_1$ & $e_1 - 2 e_2$ & $-e_1 - e_2$  \\
  \hline
$e_2$  & $-e_1 - e_2$ & $-2 e_1 + e_2$ \\
  \hline
\end{tabular}\\
\end{center}

It follows from the table that
$\Delta^2(\R_3) = \langle e_1+e_2, 3e_2 \rangle $ and $\Delta(\R_3) / \Delta^2(\R_3) \cong \mathbb{Z}_3$. Using induction we prove the following result.

\begin{proposition}\label{ass-graded-1}
The following holds for each natural number $k$:
$$
\Delta^{2k-1}(\R_3) = \langle 3^{k-1} e_1, 3^{k-1} e_2 \rangle,~~~ \Delta^{2k}(\R_3) = \langle 3^{k-1} (e_1 + e_2), 3^{k} e_2 \rangle,~\textrm{and}~ \Delta^{k}(\R_3) / \Delta^{k+1}(\R_3) \cong \mathbb{Z}_3.$$
\end{proposition}

From the preceding proposition, we obtain the infinite filtration
$$
\Delta(\R_3) \supseteq \Delta^2(\R_3) \supseteq \Delta^3(\R_3) \supseteq \ldots
$$
such that
$\cap_{n \ge 0} \Delta^n(\R_3) = \{ 0 \}$. Hence, $\Delta^n(\R_3)$ in not nilpotent, but is residually nilpotent.
\bigskip

Next, we calculate the powers of augmentation ideal of $\R_4$. First, notice that
$$
\mathbb{Z}[\R_4] = \mathbb{Z} a_0 \oplus \mathbb{Z} a_1 \oplus \mathbb{Z} a_2 \oplus \mathbb{Z} a_3
$$
and
$$
\Delta(\R_4) = \mathbb{Z} (a_1 - a_0) \oplus \mathbb{Z} (a_2 - a_0) \oplus \mathbb{Z} (a_3 - a_0).
$$

Let us set
$$
e_1: = a_1 - a_0,  ~~e_2 := a_2 - a_0, ~~e_3 := a_3 - a_0.
$$

\begin{proposition}\label{ass-graded-2}
\begin{enumerate}
\item $\Delta^2(\R_4)$ is generated as an abelian group by the set $\{e_1 - e_2 - e_3, 2 e_2\}$
and $\Delta(\R_4) / \Delta^2(\R_4) \cong \mathbb{Z} \oplus \mathbb{Z}_2$.
\item If $k > 2$, then $\Delta^k(\R_4)$ is generated as an abelian group by the set $\{2^{k-1}(e_1 - e_2 - e_3), 2^k e_2\}$
and $\Delta^{k-1}(\R_4) / \Delta^k(\R_4) \cong \mathbb{Z}_2 \oplus \mathbb{Z}_2$.
\end{enumerate}
\end{proposition}

\begin{proof}
First we consider (1). The abelian group $\Delta^2(\R_4)$ is generated by the products $e_i .e_j$. We write these product in the following table.

\begin{center}
\begin{tabular}{|c||c|c|c|}
  \hline
 $\cdot$ & $e_1$ & $e_2$ & $e_3$ \\
   \hline
     \hline
$e_1$ & $e_1 - e_2 -e_3$ & $0$ & $e_1 - e_2 -e_3$ \\
  \hline
$e_2$  & $-2 e_2$ & $0$ & $-2 e_2$ \\
  \hline
$e_3$ & $-e_1 - e_2+e_3$ & $0$ & $-e_1 - e_2+e_3$ \\
  \hline
\end{tabular}\\
\end{center}
Hence, $\Delta^2(\R_4)$ is generated by the set
$$
\{e_1 - e_2 -e_3,~~ -2 e_2,~~ -e_1 - e_2+e_3\}
$$
and $\{e_1 - e_2 - e_3, 2 e_2\}$ form a basis of $\Delta^2(\R_4)$. The fact that $\Delta^2(\R_4) / \Delta^3(\R_4)\cong \mathbb{Z}_2 \oplus \mathbb{Z}_2$ follows from the properties of abelian groups.

Next, we consider (2). To find a basis of $\Delta^3(\R_4)$, multiply the basis of $\Delta^2(\R_4)$ by the elements $e_1, e_2, e_3$. Using the multiplication table, we get
$$
(e_1 - e_2 - e_3) .e_1 = 2 (e_1 + e_2 - e_3),~~e_1. (e_1 - e_2 - e_3) = 0,~~~(e_1 - e_2 - e_3) .e_2 = e_2 .(e_1 - e_2 - e_3) = 0,
$$
$$
(e_1 - e_2 - e_3) .e_3 = 2 (e_1 + e_2 - e_3),~~e_3 .(e_1 - e_2 - e_3) = 0,~~~2 e_1 . e_2 = 0,~~~2 e_2 .e_1 = - 4 e_2,
$$
$$
2 e_2 .e_2 = 0,~~~2 e_2 .e_3 = - 4 e_2,~~~2 e_3. e_2 = 0.
$$
Hence, the set $\{2 (e_1 - e_2 - e_3), 4 e_2\}$ forms a basis of $\Delta^3(\R_4)$. We obtain the general formulas using induction on $k$.
\end{proof}

Finally, we consider  the quandle $\R_5$. Recall that
$$
\mathbb{Z}[\R_5] = \mathbb{Z} a_0 \oplus \mathbb{Z} a_1 \oplus \mathbb{Z} a_2 \oplus \mathbb{Z} a_3 \oplus \mathbb{Z} a_4,
$$
and
$$
\Delta(\R_4) = \mathbb{Z} (a_1 - a_0) \oplus \mathbb{Z} (a_2 - a_0) \oplus \mathbb{Z} (a_3 - a_0) \oplus \mathbb{Z} (a_4 - a_0).
$$
Denote
$$
e_i := a_i - a_0~\textrm{for}~  i=1, 2, 3, 4.
$$

\begin{proposition}\label{ass-graded-3}
 $\Delta^2(\R_5)$ is generated as an abelian group by $  \{e_1 - e_2 - e_4, e_2 + 2 e_4, e_3 + 3 e_4, 5 e_4 \}$ and $\Delta(\R_5) / \Delta^2(\R_5) \cong \mathbb{Z}_5$.
\end{proposition}

\begin{proof}
The abelian group $\Delta^2(\R_5)$ is generated by the products $e_i .e_j$. We write these products in the following table.

\begin{center}
\begin{tabular}{|c||c|c|c|c|}
  \hline
 $\cdot$ & $e_1$ & $e_2$ & $e_3$ & $e_4$ \\
   \hline
     \hline
$e_1$ & $e_1 - e_2 -e_4$ & $e_3 - 2 e_4$ & $-e_1 - e_4$ & $e_2 - e_3 - e_4$ \\
  \hline
$e_2$  & $-e_2 - e_3$ & $e_2 - e_3 - e_4$ & $- e_1 - e_3 + e_4$ & $e_1 - 2 e_3$ \\
  \hline
$e_3$ & $-2 e_2+e_4$ & $e_1 - e_2 - e_4$ & $-e_1 - e_2+e_3$ & $-e_2 - e_3$ \\
  \hline
$e_4$ & $-e_1 - e_2+e_3$ & $-e_1 - e_4$ & $-2e_1 + e_2$ & $-e_1 - e_3 + e_4$ \\
  \hline
\end{tabular}\\
\end{center}
It follows that $\Delta^2(\R_5)$ is generated by the elements from the above table. Further, it is not difficult to see that the elements
 $$
 \{e_1 - e_2 - e_4, e_2 + 2 e_4, e_3 + 3 e_4, 5 e_4 \}
 $$
 form a basis of $\Delta^2(\R_5)$. The fact that $\Delta(\R_5) / \Delta^2(\R_5)\cong \mathbb{Z}_5$ follows from the properties of abelian groups.
\end{proof}

The preceding results lead us to formulate the following.

\begin{conjecture}
Let $\R_n$ be the dihedral quandle and $R$ a ring.
\begin{enumerate}
\item If $n>1$ is an odd integer, then $\Delta^k(\R_n) / \Delta^{k+1}(\R_n) \cong \mathbb{Z}_n$ for all $k \geq 1$.
\item If $n>2$ is an even integer, then $|\Delta^k(\R_n) / \Delta^{k+1}(\R_n)| = n$ for all $k \geq 2$.
\end{enumerate}
\end{conjecture}
\bigskip

%%%%%%%%%%%%%%%%%%%%%%%%%%%%%%%%%%%%
%%%%%%%%%%%%%%%%%%%%%%%%%%%%%%%%%%%%

\section{Power-associativity of quandle rings}\label{section7}
We know that any trivial quandle is associative, but an arbitrary quandle, and hence its quandle ring need not be associative. In particular,  the dihedral quandle $\R_3$, and hence all dihedral quandles $\R_n$, $n \geq 3$ are not associative. Also, it is easy to see that the conjugation quandle $\Conj(G)$ is associative if and only if $G$ is abelian, i.e. $\Conj(G)$ is a trivial quandle. For core quandles, we prove the following

\begin{proposition}
A core quandle $\Core(G)$ is associative if and only if $G$ has exponent 2.
\end{proposition}

\begin{proof}
Let $a, b, c \in G$. Then, in $\Core(G)$,  we have
$$
(a.b).c = c  b^{-1}  a  b^{-1}  c~\textrm{and}~ a.(b.c) = c  b^{-1}  c  a^{-1}  c  b^{-1}  c.
$$
Now, it follows that $\Core(G)$ is associative if and only if $(c  a^{-1})^2 = 1$ i.e. $G$ has exponent 2.
\end{proof}

Recall that, a ring $R$ is called {\it power-associative}  if every element of $R$ generates an associative subring of $R$ (see \cite{Albert}). If $\T$ is a trivial quandle and $R$ an associative ring, then $R[\T]$ is associative, and hence power-associative. In general, it is an interesting question to determine the conditions under which the ring $R[X]$ is power-associative.  We investigate power-associativity of quandle rings of dihedral quandles $\R_n$. The cases $n=1,2$ are obvious. For $n=3$, we prove the following result.

\begin{proposition}\label{power-ass-1}
Let $R$ be a commutative ring with unity of characteristic not equal to 2, 3 or  5. Then the quandle ring $R[\R_3]$ is not power-associative.
\end{proposition}

\begin{proof}
Let $\R_3 = \{ a_0, a_1, a_2 \}$. Notice that, $R[\R_3]$ is a commutative ring, and hence $u^2.u = u.u^2$ for all $u \in R[\R_3]$. On the other hand, we show that the equality $(u^2.u).u = u^2.u^2$ does not hold for all $u \in R[\R_3]$.
Let $u=\alpha a_0+ \beta a_1+ \gamma a_2$, where $\alpha, \beta,\gamma \in R$. We compute
$$
u^2 = (\alpha^2 + 2 \beta \gamma) a_0 + (\beta^2 + 2 \alpha \gamma) a_1 + (\gamma^2 + 2 \alpha \beta) a_2.
$$
and
$$
u^2.u^2 = (\alpha^4 + 6 \beta^2 \gamma^2 + 4 \alpha \beta^3 + 4 \alpha \gamma^3 + 12 \alpha^2 \beta  \gamma) a_0 +
(\beta^4 + 6 \alpha^2 \gamma^2 + 4 \alpha^3 \beta + 4 \beta \gamma^3 + 12 \alpha \beta^2  \gamma) a_1 +
$$
$$
+(\gamma^4 + 6 \alpha^2 \beta^2  + 4 \alpha^3 \gamma + 4 \beta^3 \gamma + 12 \alpha \beta  \gamma^2) a_2.
$$

On the other hand
$$
u^2.u = (\alpha^3 + 2 \alpha \beta \gamma + \beta \gamma^2 + 2 \alpha \beta^2 + \gamma \beta^2 + 2 \alpha \gamma^2) a_0 + (\beta^3 + 2 \alpha^2 \beta + \alpha \gamma^2 + 2 \alpha \beta \gamma + \alpha^2 \gamma + 2 \beta \gamma^2) a_1 +
$$
$$
+(\gamma^3 + 2 \alpha \beta \gamma + \alpha \beta^2 + 2 \alpha^2 \gamma + \alpha^2 \beta + 2 \beta^2 \gamma) a_2,
$$
and
$$
(u^2.u).u =(\alpha^4 + 3 \alpha^2  \beta^2 + 3 \alpha^2 \gamma^2 + 3 \alpha \beta \gamma^2 + 3 \alpha \beta^2 \gamma + 6 \alpha^2 \beta \gamma  + 3 \beta \gamma^3 +  \alpha \beta^3 + 3 \beta^3 \gamma +  \alpha  \gamma^3) a_0 +
$$
$$
+(\beta^4 + 3 \alpha \beta \gamma^2  + 3 \alpha^2 \beta \gamma + 6 \alpha \beta^2 \gamma + 3 \alpha^2 \beta^2 + 3 \beta^2 \gamma^2  + 3 \alpha \gamma^3 +  \alpha^3 \beta + 3 \alpha^3 \gamma +  \beta  \gamma^3) a_1 +
$$
$$
+(\gamma^4 + 3 \alpha^2  \gamma^2 + 3 \beta^2 \gamma^2 + 3 \alpha \beta^2 \gamma + 3 \alpha^2 \beta \gamma + 6 \alpha \beta \gamma^2  + 3 \alpha \beta^3 +  \alpha^3 \gamma + 3 \alpha^3 \beta +  \beta^3  \gamma) a_2.
$$
It follows that $R[\R_3]$ is not power-associative.
\end{proof}

\begin{proposition}\label{power-ass-2}
Let $R$ be a commutative ring with unity of characteristic not equal to 2. Then $R[\R_n]$ is not power-associative for $n > 3$.
\end{proposition}

\begin{proof}
Let $n > 3$ and $\R_n = \{ a_0, a_1, \ldots, a_{n-1} \}$. The product in $\R_n$ is defined by the rule
$$
a_i.a_j = a_{2j-i},
$$
where the indexes are taken by modulo $n$. Consider an element $u =  a_0 + 2 a_1 \in R[\R_n]$. Then we have
$$
u^2 =  a_0 + 4 a_1 + 2 a_2 + 2 a_{n-1},
$$
and hence
$$
u^2.u  = 5 a_0 + 8 a_1 + 2  a_2 + 8 a_{n-3} + 4 a_{n-2} + 4 a_{n-1}.
$$
On the other hand
$$
u.u^2  = a_0 + 8 a_1 + 4 a_2 + 4 a_3 + 2 a_4 + 4 a_{n-3} + 2 a_{n-2} + 2 a_{n-1}.
$$
We see that if $n > 3$, then the coefficient of $a_0$ in $u^2.u$ equals to $5$, but the coefficient of $a_0$ in $u.u^2$  equals to $1$. Hence, $u^2.u  \not= u.u^2$, and the ring $R[\R_n]$ is not power-associative.
\end{proof}

We conclude with the following questions.

\begin{question}
Let $X$ and $Y$ be two racks such that $R[X]\cong R[Y]$. Does it follow that $X \cong Y$?
\end{question}

\begin{question}
Let $X$ and $Y$ be two racks with  $\mathcal{X}_R(X) \cong \mathcal{X}_R(Y)$. Does it follow that $X \cong Y$?
\end{question}

\bigskip

\medskip \noindent \textbf{Acknowledgement.} 
The results given in sections \ref{section3}, \ref{section6} and \ref{section7} are supported by the Russian Science Foundation grant 16-41-02006, and the results in \ref{section4} and \ref{section5} are supported by the DST grant INT/RUS/RSF/P-2 and SERB MATRICS grant.

\end{document}